\documentclass{article}
\usepackage{amsmath}
\usepackage{amssymb}
\usepackage{amsthm}
\usepackage[english]{babel}
\usepackage{times}
\usepackage[T1]{fontenc}
\usepackage[all]{xy}
\usepackage{geometry}

\theoremstyle{plain}
\newtheorem{teo}{Theorem}
\newtheorem{coro}[teo]{Corollary}
\newtheorem{lem}[teo]{Lemma}
\newtheorem{pro}[teo]{Proposition}

\theoremstyle{definition}

\newcommand{\B}{\mathbb{B}}
\newcommand{\HH}{\mathbb{H}}
\newcommand{\N}{\mathbb{N}}

\newcommand{\rr}{\mathbb{R}}
\newcommand{\s}{\mathbb{S}}

\newcommand{\p}{\partial}

\DeclareMathOperator{\ext}{ext}
\newcommand{\RR}{\mathbb{R}}
\newcommand{\BB}{\mathbb{B}}

\renewcommand{\SS}{\mathbb{S}}

\title{\bf The orthogonal projection on slice functions on the quaternionic sphere}
\author{Nicola Arcozzi\thanks{Partially supported by the PRIN project Real and Complex Manifolds of the Italian MIUR and by INDAM-GNAMPA}\\ 
\normalsize Dipartimento di Matematica, Universit\`a di Bologna \\ 
\normalsize Piazza di Porta San Donato 5, 40126 Bologna, Italy,  nicola.arcozzi@unibo.it \\
\and 
Giulia Sarfatti\thanks{Partially supported by INDAM-GNSAGA, by the PRIN project Real and Complex Manifolds and by the FIRB project Differential Geometry and Geometric Function Theory of the Italian MIUR}\\ 
\normalsize Dipartimento di Matematica, Universit\`a di Bologna \\ 
\normalsize Piazza di Porta San Donato 5, 40126 Bologna, Italy,  giulia.sarfatti@unibo.it \\}

\date{}

\begin{document}
\maketitle
\begin{abstract}
We study the $L^p$ norm of the orthogonal projection from the space of quaternion valued $L^2$ functions to the closed subspace of {\em slice} $L^2$ functions.
\end{abstract}

The aim of this short note is to study the orthogonal projection $\Pi$ from the space of quaternion valued $L^2$ functions to its closed subspace of {\em slice} $L^2$ functions. In particular, to compute the norm of the projection operator we will first show that we can write $\Pi$ in terms of a quaternionic {\em slice} Poisson kernel.

Let $\HH=\RR+\RR i +\RR j +\RR k$ denote the non commutative $4$-dimensional real algebra of quaternions and let $\BB=\{q\in\HH:\ |q|<1\}$ be the unit ball in $\HH$.
Its boundary $\partial \BB$ contains elements of the form $q=e^{It},\ I\in\SS,\ t\in\RR$, where $\SS=\{q\in \HH : q^2=-1\}$ is the two dimensional sphere of imaginary units in $\HH$.
We endow $\partial\BB$ with the measure $d\Sigma\left(e^{It}\right)=d\sigma(I)dt$, which is naturally associated with the Hardy space $H^2(\BB)$ 
of the slice regular functions on $\BB$, see \cite{AS}. We here normalize the measures so to have $\Sigma(\partial\BB)=\sigma(\SS)=1$. 
%\marginpar{{ Nel nostro lavoro $d\Sigma$ aveva un nome diverso, $dVol_{\p\B}$. Ce ne importa qualcosa?} }. 
The functions we are concerned with here satisfy algebraic conditions and size conditions. 
A function $f:\p\BB\to\HH$ is a 
\it slice function \rm if for any $I,J\in \s$
\begin{equation}\label{repfor}
\begin{aligned}
f(e^{Jt})&=\frac{1}{2}\left[f(e^{It})+f(e^{-It})\right]+ \frac{ JI}{2}\left[f(e^{-It})-f(e^{It})\right]. %\\
%&=\frac{1-JI}{2}f(e^{It})+\frac{1+JI}{2}f(e^{-It}).
\end{aligned}
\end{equation}
That is, slice functions are affine in the $\SS$ variable:
\begin{equation}\label{slice}
 f(e^{Jt})=a(t)+Jb(t),
\end{equation}
where $a,b:\partial\BB\to\HH$ are functions depending on $t$ alone. Condition \eqref{slice} has an algebraic nature. 
Equation \eqref{repfor} also provides a formula to extend a function $f_I$ 
defined on a slice $\p \B_I=\p \B\cap (\RR+\RR I)$ to a slice function $f:=\ext(f_I)$ defined on the entire sphere $\p\B$.

Slice functions were introduced in the more general setting of real alternative algebras by Ghiloni and Perotti in \cite{ghiloniperotti}. 
A notable class of slice functions is provided by the restriction to the unit sphere in $\HH$ of
slice regular functions. Such functions have been object of intensive research since the seminal work by Gentili and Struppa \cite{GSAdvances}. 
A concrete example is given by converging power series of the form $e^{It}\mapsto \sum_{n\in \mathbb{Z}}e^{Int}a_n$, with quaternionic coefficients $a_n$.

We refer to \cite{libroGSS} for results concerning slice regular functions and to \cite{Brackx} for functional analysis results in the quaternionic setting.

%Slice functions in the more general setting of real alternative algebras are studied in \cite{ghiloniperotti}.
We denote by $L_s=L_s(\partial\BB)$ the class
of the measurable slice functions on $\partial\BB$, which has a natural structure of right linear space on $\HH$.

We will denote by $L^p=L^p(\partial\BB)$, $1\le p\le\infty$, the space of the functions $\varphi:\partial\BB\to\HH$ such that $\int_{\partial\BB}|\varphi|^pd\Sigma=:\|\varphi\|_p^p<\infty$,
and let $L^p_s:=L^p\cap L_s$. It is not difficult to prove that $L^2_s$ is a closed subspace of $L^2$ (see Proposition \ref{chiuso} below), hence we have an orthogonal projection operator $\Pi:L^2\to L^2_s$: $\Pi$ is self-adjoint, surjective, 
and $\Pi^2=\Pi$. 

Let 
$$
\|\Pi\|_{p,p}=\sup_{L^2\ni\varphi\ne0}\frac{\|\Pi\varphi\|_p}{\|\varphi\|_p}.
$$ 
Obviously, $\|\Pi\|_{2,2}=1$.

%\section{Main result}
Our main goal here is computing the norm of $\Pi:L^p\to L^p_s$.
\begin{teo}\label{justthis}
\begin{enumerate}
 \item[(i)] $\|\Pi\|_{2,2}=1$, $\|\Pi\|_{\infty,\infty}=\frac{4}{3}$.
 \item[(ii)] Let $2\le p\le \infty$. Then, $\|\Pi\|_{p,p}\le2\left(\frac{2p-2}{3p-2}\right)^{\frac{p-1}{p}}.$
 \item[(iii)] If $1\le p,q\le \infty$ and $\frac{1}{p}+\frac{1}{q}=1$, then $\|\Pi\|_{p,p}=\|\Pi\|_{q,q}$.
\end{enumerate}
\end{teo}
The general theory implies that $\Pi$ is a contraction on $L^2$ and the equality in (iii). 
%See \cite{Brackkx} for functional analysis results in the quaternionic setting. 
The inequality in (ii) gives the right constant for $p=\infty$, but not for $p=2$.
Finding the exact value of $\|\Pi\|_{p,p}$ for all $p$'s might shed some light on the geometry of slice regular functions, and we think it is an interesting open problem.
The study presented in this paper began when discussing some problems related to Preprint \cite{S}. 

On route to the proof of Theorem \ref{justthis}, we shall write $\Pi$ as an explicit integral operator.

Let $\varphi \in L^2$. The measure $d\Sigma$ we are considering is the product of the (normalization of the) standard 
Lebesgue measure on the unit circle times the standard area element on the two sphere. 
Hence,  $\varphi$ belongs to the $L^2$ space of the unit circle $\p\B_I$ for almost every $I\in \s$, and 
we can expand $\varphi$ as a power series on almost each slice of the form
\[\varphi (e^{It})=\sum_{n=-\infty}^{\infty}e^{Int}a_n(I).\]
The coefficients depend on the imaginary unit $I\in \s$. %\marginpar{questi conti non servono alla dimostrazione del Lemma successivo?}
\begin{pro}\label{media}
Let $\varphi \in L^2$. Then the orthogonal projection of $\varphi$ to the space of slice functions is
 \[\Pi \varphi (e^{Jt})=\sum_{n=-\infty}^{\infty}e^{Jnt}\tilde{a}_n\]
for any $e^{Jt}\in \p\B$, where the coefficients are given by the integral means
\[\tilde{a}_n=\int_{\s}d\sigma(I)a_n(I)\] 
 for any $n\in \mathbb{Z}$.
 \end{pro}
\begin{proof}
Let $f(e^{It})=\sum_{n\in \mathbb{Z}}e^{Int}\alpha_n \in L^2_s$. Taking into account that, for any measurable function $g:\p\B \to \HH$, $\int_{\pi}^{2\pi}g(e^{It})dt=\int_{0}^{\pi}g(e^{-It})dt$, which implies that 
\[\int_{\p\B}g(e^{It}) d\Sigma(e^{It})=\frac{1}{2\pi}\int_{0}^{2\pi}\int_{\s}g(e^{It})d\sigma(I)dt,\] we can write  
\begin{align*}
\langle f, \varphi \rangle_{L^2}&=\frac{1}{2\pi}\int_0^{2\pi}dt\int_{\s}d\sigma(I)\overline{\varphi(e^{It})}f(e^{It})=\int_{\s}d\sigma(I)\sum_{n=-\infty}^{\infty}\overline{a_n(I)}\alpha_n\\
&=\sum_{n=-\infty}^{\infty}\overline{\int_{\s}d\sigma(I)a_n(I)}\alpha_n=\sum_{n=-\infty}^{\infty}\overline{\tilde a_n}\alpha_n=\langle f, \tilde \varphi \rangle_{L^2}
\end{align*}
where $\tilde \varphi (e^{Jt})=\sum_{n\in \mathbb{Z}}e^{Jnt}\tilde{a}_n$ is a slice function. Hence $\Pi\varphi = \tilde \varphi$.
%conti con test function in $L^2_s$
\end{proof}

We observe first that functions $f$ in $L^2_s$ extend as power series to the interior of the unit ball:

\begin{equation}\label{bilateral}
 f(re^{It})=\sum_{n=-\infty}^{+\infty}r^{|n|}e^{Int}\widehat{f}(n),\text{ where }\widehat{f}(n)=\frac{1}{2\pi}\int_{-\pi}^\pi e^{-nIs}f(e^{Is})ds\text{ is independent of }I.
\end{equation}

For $0\le r<1$, let $P^I_r$ be the Poisson kernel on the unit circle of the complex plane 
$\RR+\RR I \subset \HH$: \[P^I_r(t)=\frac{1-r^2}{|1-re^{It}|^2}=\sum_{n=-\infty}^{+\infty}r^{|n|}e^{Int}.\]
We point out that $P^I_r=P_r$ is independent of the imaginary unit $I$. In the following Lemma, we use the notation $P^I_r$ to stress the role of the imaginary unit in formula \eqref{formula2}.
\begin{lem}\label{formula}
For any $\varphi \in L^2$, the extension to the interior of $\B$ of its projection is given by
 $$
 \Pi\varphi(re^{It})=\int_{\partial\BB}K(re^{It},e^{Js}) \varphi (e^{Js}) d\Sigma(e^{Js}),
 $$
 where
 \begin{equation}\label{formula2}
 K(re^{It},e^{Js})= \frac{1}{2}\left[P^J_r(t+s)+P^J_r(t-s)\right]+\frac{IJ}{2}\left[P^J_r(t+s)-P^J_r(t-s)\right].
 \end{equation}
 
\end{lem}
\begin{proof}
Fix $J\in \s$ such that the restriction $\varphi_J$ of $\varphi$ to $\p\B_J$ belongs to the $L^2$ space of the circle $\p\B_J$, and consider the harmonic extension of $\varphi_J$ 
to the interior of the disc $\B_J$,
\[\varphi_J(re^{Js})=\sum_{n=-\infty}^{\infty}r^{|n|}e^{Jns}a_n(J).\]
The coefficients of the projection $\Pi\varphi$, recalling Proposition \ref{media}, are given by
\[\tilde a_n =\int_{\s}d\sigma(J)a_n(J)=\int_{\s}d\sigma(J)\frac{1}{2\pi}\int_{0}^{2\pi}e^{-Jns}\varphi(e^{Js})ds.\]
Hence, for any  $I\in \s, t\in \rr$,
\begin{align*}
\Pi\varphi (re^{It})&=\sum_{n=-\infty}^{+\infty}r^{|n|}e^{Int}\tilde a_n=\int_{\s}d\sigma(J)\frac{1}{2\pi}\int_{0}^{2\pi}\sum_{n=-\infty}^{+\infty}r^{|n|}e^{Int} e^{-Jns}\varphi(e^{Js})ds\\
&=\int_{\p\B}K(re^{It}, e^{Js})\varphi (e^{Js}) d \Sigma(e^{Js})
\end{align*}
where
\[K(re^{It}, e^{Js})=\sum_{n=-\infty}^{+\infty}r^{|n|}e^{Int} e^{-Jns}.\]
To conclude, notice that $K(re^{It}, e^{Js})$ is a slice function with respect to the variable $e^{It}$, hence it can be written as %\marginpar{mi torna diverso dalla tesi!!!}
\begin{align*}
K(re^{It}, e^{Js})&=\frac{1}{2}\left(K(re^{Jt}, e^{Js})+K(re^{-Jt}, e^{Js})\right)+\frac{IJ}{2}\left(K(re^{-Jt}, e^{Js})-K(re^{Jt}, e^{Js})\right)\\
&=\frac{1}{2}\left(\sum_{n=-\infty}^{+\infty}r^{|n|}e^{Jn(t-s)}+\sum_{n=-\infty}^{+\infty}r^{|n|}e^{Jn(-t-s)}\right)\\
&\hskip 3 cm +\frac{IJ}{2}\left(\sum_{n=-\infty}^{+\infty}r^{|n|}e^{Jn(-t-s)}-\sum_{n=-\infty}^{+\infty}r^{|n|}e^{Jn(t-s)}\right)\\
&=\frac{1}{2}\left(P_r^J(t-s)+P_r^J(-t-s))\right)+\frac{IJ}{2}\left(P^J_r(-t-s)-P^J_r(t-s)\right).
\end{align*}
\end{proof}

\begin{coro}\label{bene}
 We have that $\Pi\varphi(re^{It})=A(r,t)+I B(r,t)$, where
% $$
% A(r,t)=\int_\SS d\sigma(J)\frac{1}{2\pi}\int_{-\pi}^\pi \frac{1}{2}\left[P^J_r(t+s)+P^J_r(t-s)\right]\varphi(e^{Js})ds 
% $$
% and
$$
 A(r,t)=\frac{1}{2\pi}\int_{-\pi}^\pi \frac{1}{2}\left[P_r(t+s)+P_r(t-s)\right]\left\{\int_\SS\varphi(e^{Js})d\sigma(J)\right\}ds
 $$
and
 $$
  B(r,t)=\frac{1}{2\pi}\int_{-\pi}^\pi \frac{1}{2}\left[P_r(t+s)-P_r(t-s)\right]\left\{\int_\SS J\varphi(e^{Js})d\sigma(J)\right\}ds.
 $$
% $$
%  B(r,t)=\int_\SS d\sigma(J) \frac{1}{2\pi}\int_{-\pi}^\pi \frac{1}{2}\left[P^J_r(t+s)-P^J_r(t-s)\right] J\varphi(e^{Js})ds.
% $$
\end{coro}
In $A$ and $B$ the function $\varphi$ enters, respectively, through its ``imaginary mean'' and ``imaginary first moment'' on each copy of $\SS$ inside $\partial\BB$.

Passing in the limit as $r\to1^-$, using the fact that the classical Poisson kernel as linear operator tends to the Dirac delta function, we obtain -at least when $\varphi\in C=C(\partial\BB,\HH)$ is continuous on $\partial\BB$-:
\begin{equation}\label{meglio}
 \Pi\varphi(e^{tI})=\int_\SS(1-IJ)\varphi(e^{Jt})d\sigma(J).
\end{equation}
By density of $C$ in $L^p$ ($p<\infty$), the integral formula for $\Pi$ extends to $L^p$ if $\Pi$ is bounded on $L^p$. Since $\Pi$ is bounded on $L^2$, by inclusion 
it is defined on  $L^p$'s at least when $2\le p\le\infty$. Formula \eqref{meglio} is in this case true for $a.e.$ $e^{tI}$ in $\partial\BB$.

Consider the case $p=\infty$ of the theorem first. Taking the essential supremum, we have that
\begin{align}\label{ottimo}
 &\|\Pi\varphi\|_\infty\le\int_\SS|1-IJ|d\sigma(J)\|\varphi\|_\infty,
% &\text{ with equality if }\varphi(e^{Js})=|1-I_0J|(1-I_0J)^{-1}\text{ for some }I_0\in\SS. \nonumber
\end{align}
with equality if $\varphi(e^{Js})=|1-I_0J|(1-I_0J)^{-1}$  for some $I_0\in\SS$. Using spherical coordinates on $\SS$  such that $I=(0,0,1)$ and $J=(\sin t \cos s, \sin t \sin s, \cos t )$, with $s \in (0,2\pi)$ and $t \in (0,\pi)$, we get that $|1-IJ|=2\sin(t/2)$, $d\sigma(J)=\sin(t)dtds$,
and hence that  
\begin{equation}\label{sferica}
\int_\SS|1-IJ|d\sigma(J)=4/3,
\end{equation}
thus concluding the proof of Part $(i)$ of Theorem \ref{justthis}. 
Similarly, when $2\le p<\infty$ 
%, possibly on a dense subspace only for $1< p<2$, 
we have, with $\frac{1}{p}+\frac{1}{q}=1$,
\begin{equation}\label{dipiu}
 \|\Pi\varphi\|_p\le\left(\int_\SS|1-IJ|^qd\sigma(J)\right)^{1/q}\|\varphi\|_p.
\end{equation}
We have used H\"older's inequality in \eqref{meglio} and the fact that the first factor on the right of \eqref{dipiu} is independent of $I$.
%Finally, for $p=1$ we have
%\begin{equation}\label{basta}
% \|\Pi\varphi\|_1\le2\|\varphi\|_1.
%\end{equation}
A simple calculation with the same spherical coordinates on $\SS$ used to obtain \eqref{sferica} gives
$$
\left(\int_\SS|1-IJ|^qd\sigma(J)\right)^{1/q}=2\left(\frac{2}{q+2}\right)^{1/q}=2\left(\frac{2p-2}{3p-2}\right)^{\frac{p-1}{p}},
$$
and the proof of Theorem \ref{justthis} is finished. 

\vskip.5cm

\begin{pro}\label{chiuso}
$L^2_s$ is a closed subspace of $L^2$.
\end{pro}
\begin{proof}
Let $\{f_n\}_n$ be a sequence in $L^2_s$ converging in $L^2$ norm to $f\in L^2$. We want to show that the difference $f(e^{Jt})-(a(t)+Jb(t))$,  where $a(t)=\frac{1}{2}[f(e^{It})+f(e^{-It})]$ and $b(t)=\frac{I}{2}[f(e^{-It})-f(e^{It})]$ for any $I\in \s$, equals zero in $L^2$ norm.
For any $n\in \N$ we have
\begin{align*}
&\|f(e^{Jt})-\left(a(t)+Jb(t)\right)\|_{2}\\
&\le \|f(e^{Jt})-f_n(e^{Jt})\|_{2}+\|f_n(e^{Jt})-a_n(t)-Jb_n(t)\|_{2}+\|a_n(t)+Jb_n(t)-a(t)-Jb(t)\|_{2},
\end{align*}
with obvious notation. 
The first summand tends to $0$ as $n$ goes to $\infty$ and the second summand vanishes for any $n$ since $f_n\in L^2_s$. The last term can be bounded as   
\begin{align*}
\|a_n(t)+Jb_n(t)-a(t)-Jb(t)\|_{2}\le \|a_n(t)-a(t)\|_{2}+\|b_n(t)-b(t)\|_{2},
\end{align*}
where we can choose any slice $\p\B_I$ to compute both quantities on the right side.
The fact that $f_n$ converges to $f$ in $L^2$ norm yields that, for $\sigma$-almost every $I$,  $\int_{0}^{2\pi}|f_n(e^{It})-f(e^{It})|^2dt$ tends to $0$ as $n$ goes to $\infty$, and hence, that both $\|a_n(t)-a(t)\|_{2}$ and $\|b_n(t)-b(t)\|_{2}$ tend to $0$ as $n$ goes to $\infty$. Therefore we conclude that 
$f\in L^2_s$.
%The fact that for $\sigma$-almost every $I$ $\int_{0}^{2\pi}|f_n(e^{It})-f(e^It)|^2dt$ 
%
%\[\|a_n(t)-a(t)\|^2_{2}\le \frac{1}{4}\|f_n(e^{It})-f(e^{It})\|^2_{2}+ \frac{1}{4}\|f_n(e^{-It})-f(e^{-It})\|^2_{2}=\frac{1}{2}\int
\end{proof}

We end the note with some remarks.
It is a peculiar fact that $1<\|\Pi\|_{\infty,\infty}<\infty$. Many natural projection operators defined on $L^2$ spaces extend to contractions to $L^p$ for $1<p<\infty$. 
This is the case with restriction operators $\varphi\mapsto \chi_E \varphi$, where $\chi_E$ is a measurable characteristic function; or, closer to the operator we are considering here,
the projection onto the space of functions endowed with some symmetry. For instance, let $L_\SS$ be the space of the functions $\varphi$ on $\partial\BB$ which are measurable and such that
$\varphi(e^{It})=a(t)$ depends on $t$ alone, set $L_\SS^p=L^p\cap L_\SS$, and let $\Pi_\SS:L^2\to L_\SS^2$ be the orthogonal projection. We might write
$$
\Pi_\SS\varphi(e^{It})=\int_{SO(3)}\varphi(e^{RIt})dh(R),
$$
where $SO(3)$ is the orthogonal group, acting on $\SS$ as $I\mapsto RI$, for any $R\in SO(3)$, and $dh$ is its (bi-invariant) Haar measure, normalized to have $h(SO(3))=1$.
Alternatively, we might ask $\varphi$ to be covariant under rotations, 
\begin{equation}\label{lagoscuro}
\varphi(e^{RIt})=R\varphi(e^{It}).
\end{equation}
(If $\varphi=e^{Ks}$, then $R\varphi=e^{RKs}$.)
The projection from $L^2$ to the subspace of the functions satisfying \eqref{lagoscuro} is
\begin{equation}\label{anita}
\Pi^\prime_\SS\varphi(e^{It})=\int_{SO(3)}R^{-1}\varphi(e^{R(I)t})dh(R).
\end{equation}
Observe that there is a formal analogy between \eqref{lagoscuro} and \eqref{anita}, and the integrals appearing in Corollary \ref{bene}.

At the other extreme, some projection operators on $L^2$ are unbounded in $L^\infty$. The prototype is the conjugate function operator in one complex variable, and an avatar of it in our
context might be
$$
\varphi (e^{It})=\sum_{n=-\infty}^{\infty}e^{Int}a_n(I)\ \mapsto\ \tilde{\Pi}\varphi (e^{It})=\sum_{n=0}^{\infty}e^{Int}a_n(I).
$$
Tipically, projections of this kind can be expressed in terms of singular integral operators.

The projection $\Pi$ onto slice functions seems to be intermediate between these two families of operators. 
It shares with the first the fact that it is linked with a geometric-algebraic invariance property. This might ``explain'' the boundedness in $L^\infty$. On the other hand, the invariance 
can not be simply stated in measure-theoretic terms, and this might ``explain'' why it is not a contraction of $L^\infty$: cancellations in the integrals play a role, and this causes the
$L^p$ norm of the operator to grow.

To have a hint as to which kind of cancellations are involved, it is instructive to consider the relation:
$$
\int_\SS|\Pi\varphi(e^{It})|^2d\sigma(I)=\left|\int_\SS\varphi(e^{It})d\sigma(I)\right|^2+\left|\int_\SS I\varphi(e^{It})d\sigma(I)\right|^2,
$$
which follows from Corollary \ref{bene}.
If the function $\varphi(e^{Jt})=a(t)+Jb(t)$ is already slice, then the first integral vanishes for $a=0$, while the second does for $b=0$. 
This is somehow a measure of the amount of cancellation which is going on in the integrals.

%\section*{References}

\end{document}